\documentclass[11pt]{amsart}
\usepackage{graphicx}
\usepackage[centertags]{amsmath}
\usepackage{amsfonts}
\usepackage{amssymb}
\usepackage{amsthm}
\usepackage{newlfont}
\usepackage{pgfplots}
\pgfplotsset{compat=1.14}
\newtheorem{theorem}{Theorem}[section]
\newtheorem{corollary}[theorem]{Corollary}

\theoremstyle{definition}
\newtheorem{defn}[theorem]{Definition}
\theoremstyle{remark}

\newcommand{\id}{\textrm{id}}
\title[On an indefinite metric on a Riemannian manifold]{Spheres and circles with respect to an indefinite metric on a Riemannian manifold with circulant structures}

\begin{document}
\author{Georgi Dzhelepov}
\address{Agricultural University of Plovdiv\\ Department of Mathematics, Informatics and Physics\\ 12 Mendeleev Blvd., 4000 Plovdiv, Bulgaria}
\email{dzhelepov@au-plovdiv.bg}

\begin{abstract}
We consider a $3$-dimensional differentiable manifold with two circulant structures -- a Riemannian metric and an additional structure, whose third power is the identity. The structure is compatible with the metric such that an isometry is induced in any tangent space of the manifold. Further, we consider an associated metric with the Riemannian metric, which is necessary indefinite. We find equations of a sphere and of a circle, which are given in terms of the associated metric, with respect to the Riemannian metric.\end{abstract}

%%%%
%%%%

\maketitle

\section{Introduction}
\label{intro}

The study of Riemannian manifolds with additional structures is very substantial topic in modern differential geometry. Some of the manifolds are equipped with a structure, which satisfies an equation of third power (for example \cite{22}, \cite{dok}, \cite{Puj} and \cite{Yano}).

 The models of a sphere and the relations between spheres and hyperboloids, between spheres and cones are of certain interest in pseudo-Rie\-mannian geometry. Another current problem is the obtaining of the correspondence between a circle and other quadratic curves (for instance  \cite{4}, \cite{6}, \cite{8}, \cite{2}, \cite{10} and \cite{12}).

The object of the present paper is a consideration of a $3$-dimensional differentiable manifold $M$ equipped with two circulant structures -- a Riemannian metric $g$ and a tensor $q$ of type $(1, 1)$, whose third power is the identity and $q$ acts as an isometry on $g$, i.e. $(M, g, q)$. We investigate an associated metric $f$,  which is introduced in \cite{3}. The metric $f$ is necessary indefinite and it determines spacelike vectors, null vectors and timelike vectors in the tangent space $T_{p}M$, $p\in M$.

The paper is organized as follows. In Sect.~\ref{sec:2}, we recall necessary facts about the $3$-dimensional Riemannian manifold $(M, g, q)$ and about a $q$-basis of $T_{p}M$. In Sect.~\ref{sec:3}, we consider the properties of the associated metric $f$ on $(M, g, q)$. In Sect.~\ref{sec:4}, we obtain equations of spheres, which are given in terms of $f$, with respect to $g$. Also we establish that every vector from an othonormal $q$-basis of $T_{p}M$ is a null one with respect to $f$. In Sect.~\ref{sec:5}, we find equations of circles, which are given in terms of $f$,  with respect to $g$.

\section{Preliminaries}\label{sec:2}
 Let $M$ be a $3$-dimensional differentiable manifold equipped with a tensor $q$ of type $(1, 1)$, which acts as an isometry in every tangent space $T_{p}M$, $p\in M$. Let the coordinates of $q$ with respect to a local coordinate system form the circulant matrix
 \begin{equation}\label{f1}
    q=\begin{pmatrix}
      0 & 1 & 0 \\
      0 & 0 & 1 \\
      1 & 0 & 0 \\
    \end{pmatrix}.
\end{equation}
Due to \eqref{f1} we have
\begin{equation}\label{q3}
    q^{3}= \id.
\end{equation}
Let $g$ be a positive definite metric on $M$ such that
\begin{equation}\label{2.12}
     g(qu, qv)=g(u,v).
\end{equation}
In the latter equality and further, $u, v, w$ will stand for arbitrary elements of the algebra on the smooth vector fields on $M$ or vectors in $T_{p}M$.

Equalities \eqref{f1} and \eqref{2.12} imply that the local coordinates of $g$ form a circulant matrix.
We denote by $(M, g, q)$ the manifold $M$ equipped with the metric $g$ and the structure $q$, which are defined by (\ref{f1}) and (\ref{2.12}).

It is well known that the norm of a vector $u$ is
   \begin{equation}\label{size}
   \|u\|=\sqrt{g(u, u)},
   \end{equation}
and for the angle $\varphi=\angle(u, qu)$ we have
  \begin{equation}\label{cos}\cos\varphi=\frac{g(u, qu)}{g(u, u)}\ .\end{equation}

In \cite{1},  for $(M, g, q)$ it is verified that $\varphi$ is in $\big[0,\frac{2\pi}{3}\big]$.
  If $\varphi\in \big(0,\frac{2\pi}{3}\big)$ and $u\in T_{p}M$, then $\{u, qu, q^{2}u
\}$ form a basis, which is called a $q$-basis. There exists an orthonormal $q$-basis of $T_{p}M$.

\section{Properties of the associated metric $f$}\label{sec:3}

The associated metric $f$ on $(M, g, q)$, determined in \cite{3} by
    \begin{equation}\label{defF}
   f(u, v)=g(u, qv)+g(qu, v),
   \end{equation}
 is necessary indefinite.
Using \eqref{f1} and \eqref{2.12}, we establish that $f$ satisfies the following equalities:
\begin{equation}\label{ai}
  f(u, u)=2g(u, qu),
\end{equation}
\begin{equation}\label{bi}
  f(u, qu)=g(u, u)+g(u, qu).
\end{equation}
According to the physical terminology we give
\begin{defn}\label{D1}
Let $f$ be the associated metric on $(M,g,q)$. If a vector $u$ satisfies $f(u, u)>0$ (resp. $f(u, u)<0$), then $u$ is spacelike (resp. timelike).
If a nonzero vector $u$ satisfies $f(u, u)=0$, then $u$ is null.
\end{defn}
From \eqref{size}, \eqref{cos} and \eqref{ai} we get
$f(u, u)=2\|u\|^{2}\cos\varphi.$
Then, having in mind Definition~\ref{D1}, we obtain the following
\begin{theorem}\label{thmR}
Let $f$ be the associated metric on $(M,g,q)$. If $\varphi$ is the angle between $u$ and $qu$, then the following
propositions hold:
\begin{itemize}
  \item[(i)] $u$ is a spacelike vector if and only if $\varphi\in \big[0,\frac{\pi}{2}\big)$;
  \item[(ii)] $u$ is a null vector if and only if $\varphi=\frac{\pi}{2}$;
  \item[(iii)] $u$ is a timelike vector if and only if $\varphi\in \big(\frac{\pi}{2},\frac{2\pi}{3}\big)$.
\end{itemize}
\end{theorem}
Obviously, taking into account \eqref{q3} and \eqref{ai}, we have
\begin{corollary}
If $u$ is a spacelike (null or timelike) vector, then $qu$ and $q^{2}u$ are spacelike (null or timelike) vectors, respectively.
   \end{corollary}

\section{Spheres with respect to $f$}\label{sec:4}

Let $\{u, qu, q^{2}u\}$ be an orthonormal $q$-basis of $T_{p}M$. Let $p_{xyz}$ be a coordinate system such that the vectors $u$, $qu$ and $q^{2}u$ be on the axes $p_{x}$, $p_{y}$ and $p_{z}$, respectively. So $p_{xyz}$ is an orthonormal coordinate system.
If $N(x, y, z)$ is an arbitrary point with a radius vector $v$, then $v$ is expressed by
\begin{equation}\label{V1}
  v=xu+yqu+zq^{2}u.
\end{equation}
The equation of a sphere $s$ of a radius $r$ centered at the origin $p$, with respect to the associated metric $f$ on $(M,g,q)$, is
\begin{equation}\label{V2}
s:\quad  f(v, v)=r^{2}.
\end{equation}
Having in mind that $f$ is an indefinite metric, we have three different options for the constant $r^{2}$, which are $r^{2}>0$, $r^{2}=0$ and $r^{2}<0$.

We apply \eqref{V1} into \eqref{V2}, and using \eqref{ai} and \eqref{bi}, we obtain the equation of $s$ with respect to $g$ as follows
\begin{equation}\label{hyp1}
  2xy+2xz+2yz=r^{2}.
\end{equation}
 We rotate the coordinate system $p_{xyz}$ into $p_{x^{'}y^{'}z^{'}}$ by substitutions:
\begin{eqnarray}\label{rot} \nonumber
   x &=& \frac{1}{\sqrt{2}}x^{'}-\frac{1}{\sqrt{6}}y^{'} +\frac{1}{\sqrt{3}}z^{'},\\
  y &=& \frac{\sqrt{2}}{\sqrt{3}}y^{'} +\frac{1}{\sqrt{3}}z^{'},\\\nonumber
   z &=& -\frac{1}{\sqrt{2}}x^{'}-\frac{1}{\sqrt{6}}y^{'} +\frac{1}{\sqrt{3}}z^{'}.
\end{eqnarray}
Thus, from \eqref{hyp1} we obtain the following equation of a quadratic surface
\begin{equation}\label{hyprot}
  x^{'2}+y^{'2}-2z^{'2}=-r^{2}.
\end{equation}

In particular, if $r^{2}=0$, then \eqref{hyprot} implies the equation of a cone:
\begin{equation}\label{hyp2}
 s_{0}:\quad    x^{'2}+y^{'2}-2z^{'2}=0.
\end{equation}

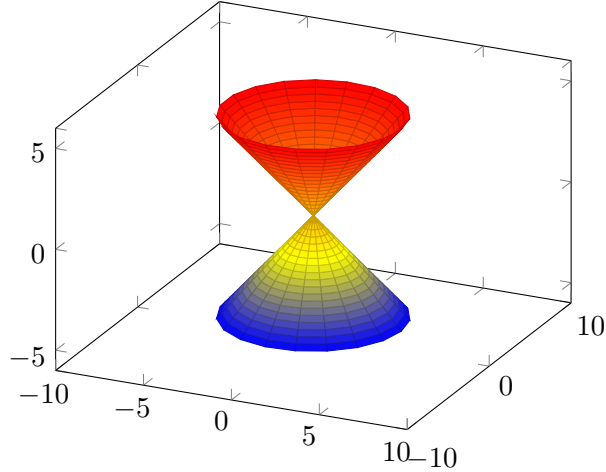
\begin{center}
\begin{tikzpicture}
\begin{axis}[
    domain=0:5,
    y domain=0:2*pi,
    xmin=-10,
    xmax=10,
    ymin=-10,
    ymax=10,
    samples=20]
\addplot3 [surf,z buffer=sort]
    ({x*cos(deg(y))},
     {x*sin(deg(y))},
     {x});
\addplot3 [surf,z buffer=sort]
    ({x*cos(deg(y))},
     {x*sin(deg(y))},
     {-x});
\end{axis}
\end{tikzpicture}
\end{center}

If $r^{2}>0$, then from \eqref{hyprot} we have a hyperboloid of two sheets:
\begin{equation}\label{hyp3}
 s_{1}:\quad    x^{'2}+y^{'2}-2z^{'2}=-a^{2},\quad a^{2}=r^{2}>0.
\end{equation}

An example, with the equation $x^{'2}+y^{'2}-2z^{'2}=-2$, is shown in the following figure.

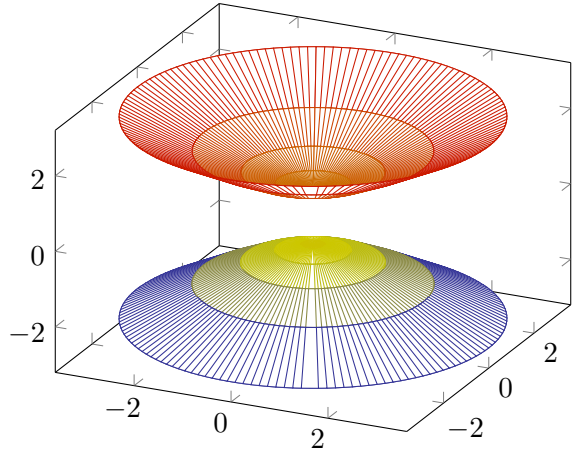
\begin{center}
\begin{tikzpicture}\label{fig2}
 \begin{axis}
  \addplot3[surf,
            fill=white,
            samples=10,
            samples y=72,
            domain=-2:2,
            y domain=0:360,
            z buffer=sort]
            ( {sinh(x)*cos(y)}, {sinh(x)*sin(y)}, {(0.71)*cosh(x)} );
            \addplot3[surf,
            fill=white,
            samples=10,
            samples y=72,
            domain=-2:2,
            y domain=0:360,
            z buffer=sort]
            ( {sinh(x)*cos(y)}, {sinh(x)*sin(y)}, -{(0.71)*cosh(x)} );
 \end{axis}
 \end{tikzpicture}
\end{center}

If $r^{2}<0$ in \eqref{hyprot}, then we obtain a hyperboloid of one sheet:
\begin{equation}\label{hyp4}
 s_{2}:\quad    x^{'2}+y^{'2}-2z^{'2}=a^{2},\quad a^{2}=-r^{2}>0.
\end{equation}

An example, with the equation $x^{'2}+y^{'2}-2z^{'2}=1$, is shown in the next figure.

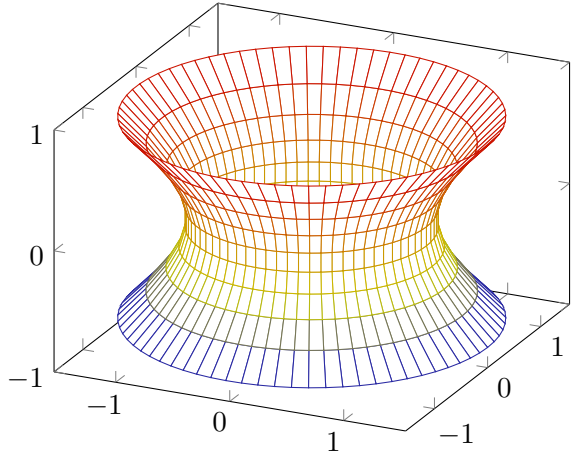
\begin{center}
\begin{tikzpicture}\label{fig1}
 \begin{axis}
  \addplot3[surf,
            fill=white,
            samples=10,
            samples y=72,
            domain=-1:1,
            y domain=0:360,
            z buffer=sort]
            ( {cosh(x)*cos(y)}, {cosh(x)*sin(y)}, {(0.71)*sinh(x)} );
 \end{axis}
\end{tikzpicture}
\end{center}

Therefore, we state the following

\begin{theorem}
Let $f$ be the associated metric on $(M,g,q)$ and $\{u, qu, q^{2}u\}$ be an orthonormal $q$-basis of $T_{p}M$. Let $p_{xyz}$ be a coordinate system such that $u\in p_{x}$, $qu\in p_{y}$, $q^{2}u\in p_{z}$ and $p_{x^{'}y^{'}z^{'}}$ be a coordinate system obtained by the rotation \eqref{rot} of $p_{xyz}$. If a sphere $s$ is given by \eqref{V2}, then the equation of $s$ with respect to $p_{x^{'}y^{'}z^{'}}$  is \eqref{hyprot}. Moreover,
\begin{itemize}
  \item[(i)] if $r^{2}=0$, then $s$ is a circular double cone $s_{0}$ with \eqref{hyp2},

  \item[(ii)] if $r^{2}>0$, then $s$ is a circular hyperboloid of two sheets $s_{1}$ with \eqref{hyp3},

  \item[(iii)] if $r^{2}<0$, then $s$ is a circular hyperboloid of one sheet $s_{2}$ with \eqref{hyp4}.
  \end{itemize}
\end{theorem}
Consequently, having in mind Definition~\ref{D1} and \eqref{V2}, the following assertion is valid.

  \begin{corollary}\label{i} If the surfaces $s_{0}$, $s_{1}$ and $s_{2}$ are produced from the sphere $s$, then
\begin{itemize}
  \item[(i)] every point on $s_{0}$ has a null radius vector,

  \item[(ii)] every point on $s_{1}$ has a spacelike radius vector,

  \item[(iii)] every point on $s_{2}$ has a timelike radius vector.
  \end{itemize}
\end{corollary}

\begin{theorem}
Let $f$ be the associated metric on $(M,g,q)$ and $\{u, qu, q^{2}u\}$ be an orthonormal $q$-basis of $T_{p}M$. If $p_{xyz}$ is a coordinate system such that $u\in p_{x}$, $qu\in p_{y}$ and $q^{2}u\in p_{z}$, then $u$, $qu$ and $q^{2}u$ are null vectors and their heads lie at the circles
\begin{equation}\label{circles}
   x^{'2}+y^{'2}=\frac{2}{3}\ ,\quad
   z^{'}=\pm\frac{1}{\sqrt{3}}\ ,
\end{equation}
where $p_{x^{'}y^{'}z^{'}}$ is the coordinate system obtained by the rotation \eqref{rot} of $p_{xyz}$.
\end{theorem}

\begin{proof}
Due to \eqref{ai} and Definition~\ref{D1} we have that $u$, $qu$ and $q^{2}u$ are null vectors with respect to $f$. Therefore, their heads are on the cone $s_{0}$ with \eqref{hyp2}. On the other hand, the heads of $u$, $qu$ and $q^{2}u$ lie at the unit sphere
\begin{equation}\label{sphere}
   x^{2}+y^{2}+z^{2}=1.
\end{equation}
The system of \eqref{hyp2} and \eqref{sphere} determines the intersection of a cone and a sphere. This intersection, with respect to the coordinate system $p_{x^{'}y^{'}z^{'}}$, is represented by
\begin{equation*}
   x^{'2}+y^{'2}-2z^{'2}=0,\quad
   x^{'2} + y^{'2} + z^{'2}=1,
\end{equation*}
or by the equivalent system \eqref{circles}, which is an intersection of a cylinder and a plane. The resulting curves are two circles.
\end{proof}

\section{Circles with respect to $f$}\label{sec:5}

Let $\alpha$ be the $2$-plane spanned by vectors $u$ and $qu$ ($qu\neq u$) in $T_{p}M$. Then, for $\varphi=\angle(u, qu)$ we have $\varphi\in(0, \frac{2\pi}{3}]$. Supposing that $\|u\|=1$, we define another vector $w$ by the equality
\begin{equation}\label{defw}
 w=\frac{1}{\sin\varphi}(-u\cos\varphi+qu),\quad \varphi=\angle(u, qu).
\end{equation}
We construct a coordinate system $p_{xy}$ on $\alpha$, such that $u$ is on the axes $p_{x}$ and $w$ is on the axes $p_{y}$.
Using \eqref{size}, \eqref{cos} and \eqref{defw}, we calculate that $g(u, w)=0$ and $g(w, w)=1$, i.e. $p_{xy}$ is an orthonormal coordinate system of $\alpha$.

Let $N(x, y)$ be a point on $\alpha$ and its radius vector is denoted by $v$. Then $v$ is expressed by the equality
\begin{equation}\label{R}
  v=xu+yw.
\end{equation}

A circle $k$ in $\alpha$ of a radius $r$ centered at the origin $p$, with respect to the associated metric $f$ on $(M,g,q)$, is determined by
\begin{equation}\label{k}
 k:\quad   f(v, v)=r^{2}.
\end{equation}
Bearing in mind that $f$ is an indefinite metric, for $r$ we have the following options: $r^{2}>0$, $r^{2}=0$ or $r^{2}<0$.

\begin{theorem}
Let $f$ be the associated metric on $(M,g,q)$ and $\alpha=\{u, qu\}$ be a $2$-plane in $T_{p}M$. Let the vector $w$ be defined by \eqref{defw} and $p_{xy}$ be a coordinate system such that $u\in p_{x}$ and $w\in p_{y}$. If a circle $k$ in $\alpha$ is defined by \eqref{k}, then the equation of $k$ with respect to $g$ is
\begin{equation}\label{18}
 (\cos\varphi)x^{2}+\frac{(1-\cos\varphi)(1+2\cos\varphi)}{\sin\varphi}xy-\frac{\cos^{2}\varphi}{1+\cos\varphi}y^{2}=\frac{r^{2}}{2}\ ,
\end{equation}
where $\varphi\in(0, \frac{2\pi}{3}]$.
\end{theorem}

\begin{proof}
From \eqref{R} and \eqref{k} we get
\begin{equation}\label{20}
  x^{2}f(u, u) +2xyf(u, w)+y^{2}f(w,w)=r^{2}.
\end{equation}
On the other hand, using \eqref{defF} -- \eqref{bi} and \eqref{defw}, we calculate
\begin{equation*}
\begin{split}
  f(u, u)&=2\cos\varphi,\quad f(w, w)=-\frac{2\cos^{2}\varphi}{1+\cos\varphi},\\
  f(u, w)&=\frac{1}{\sin\varphi}(1+\cos\varphi-2\cos^{2}\varphi).
  \end{split}
\end{equation*}
Then, from \eqref{20} we obtain \eqref{18}.
\end{proof}
Due to \eqref{18} we get the following
\begin{corollary}
The discriminant $D$ of the left side of \eqref{18} is the function
\begin{equation}\label{22}
  D=\frac{1+3\cos\varphi}{1-\cos\varphi}.
\end{equation}
\end{corollary}
\begin{corollary}
The curve $k$ is a hyperbola if and only if $\varphi\in\big(0, \arccos(-\frac{1}{3})\big)$. If $\varphi=\frac{\pi}{2}$, then the hyperbola has the equation $xy=\frac{r^{2}}{2}$.
\end{corollary}
\begin{proof} The condition for \eqref{18} to be hyperbola is $D>0$. Then from \eqref{22} the proof follows.
\end{proof}
Let $\varphi=\arccos(-\frac{1}{3})$. Then \eqref{18} implies the equation
\begin{equation}\label{24}
(\sqrt{2}x-y)^{2}=-3r^{2}.
\end{equation}
\begin{corollary} The curve $k$ has the equation~\eqref{24} if and only if $\varphi=\arccos(-\frac{1}{3})$. In particular,
\begin{itemize}
  \item[(i)]  if $r^{2}>0$, then $k$ hasn't got real points;
  \item[(ii)] if $r^{2}=0$, then $k$ is a straight line with the equation $y=\sqrt{2}x$;
  \item[(iii)] if $r^{2}<0$, then $k$ separates into two parallel lines with the equations $\sqrt{2}x-y=\pm\sqrt{-3r^{2}}$.
\end{itemize}
\end{corollary}

\begin{corollary} The curve $k$ is an ellipse if and only if $\varphi\in\big(\arccos(-\frac{1}{3}),\frac{2\pi}{3}\big)$.
\end{corollary}
\begin{proof} The condition for \eqref{18} to be hyperbola is $D<0$. Then from \eqref{22} the proof follows.
\end{proof}
Let $\varphi=\frac{2\pi}{3}$. Then \eqref{18} implies the equation
\begin{equation}\label{25}
  x^{2}+y^{2}=-r^{2}.
\end{equation}
\begin{corollary} The curve $k$ has the equation~\eqref{25} if and only if $\varphi=\frac{2\pi}{3}$. In particular,
\begin{itemize}
  \item[(i)]  if $r^{2}>0$, then $k$ hasn't got real points;
  \item[(ii)] if $r^{2}=0$, then $k$ is the origin $p$;
  \item[(iii)] if $r^{2}<0$, then $k$ is a circle.
\end{itemize}
\end{corollary}

%\begin{acknowledgements}
%\end{acknowledgements}



\end{document}